\documentclass[preprint,12pt]{amsart}
\usepackage{amsmath,amsthm,amsfonts}
\input{amssym.def}
\input{amssym.tex}

\usepackage{graphicx, color}
\numberwithin{equation}{section} 
\theoremstyle{plain}

\usepackage{color}

\usepackage{graphics}
\newcommand{\C}{\mathbb C}
\newcommand{\R}{\mathbb R}
\newcommand{\Z}{\mathbb Z}

\newtheorem{thm}{Theorem}[section]
\newtheorem{lem}[thm]{Lemma}
\newtheorem{pro}[thm]{Proposition}

\newtheorem{rk}[thm]{Remark}

\begin{document}

\title[Volume preserving diffeomorphisms ]
{\bf Volume preserving diffeomorphisms with inverse shadowing }

\author{Manseob Lee
 }
\address
     { Manseob Lee :  Department of Mathematics \\
       Mokwon University, Daejeon, 302-729, Korea
      }
\email{lmsds@mokwon.ac.kr  }

\thanks{
\newline
\indent This research was supported by Basic Science Research
Program through the National Research Foundation of Korea(NRF)
funded by the Ministry of Education, Science and Technology (No.
2011-0007649).}

\begin{abstract}
Let $f$ be a volume-preserving diffeomorphism of a closed
$C^{\infty}$ $n$-dimensional Riemannian manifold $M.$ In this
paper, we prove the equivalence between the following conditions:
\begin{itemize}
\item[(a)] $f$ belongs to the $C^1$-interior of the set of
volume-preserving diffeomorphisms which satisfy the inverse
shadowing property with respect to the continuous methods,

\item[(b)] $f$ belongs to the $C^1$-interior of the set of
volume-preserving diffeomorphisms which satisfy the weak inverse
shadowing property with respect to the continuous methods,
\item[(c)] $f$ belongs to the $C^1$-interior of the set of
volume-preserving diffeomorphisms which satisfy the orbital
inverse shadowing property with respect to the continuous methods,
\item[(d)] $f$ is Anosov.

\end{itemize}
\end{abstract}

\maketitle


\section{Introduction}

Let $M$ be a closed $C^{\infty}$ $n$-dimensional Riemannian
manifold, and let ${\rm Diff}(M)$ be the space of diffeomorphisms
of $M$ endowed with the $C^1$-topology. Denote by $d$ the distance
on $M$ induced from a Riemannian metric $\|\cdot\|$ on the tangent
bundle $TM$. Let $f:M\to M$ be a diffeomorphism, and let
$\Lambda\subset M$ be a closed $f$-invariant set.

 For $\delta>0$, a sequence of points
$\{x_i\}_{i=a}^{b}(-\infty\leq a< b \leq \infty)$ in $M$ is called
a {\it $\delta$-pseudo orbit} of $f$ if $d(f(x_i),
x_{i+1})<\delta$ for all $a\leq i\leq b-1.$
 We say that $f$ has the {\it shadowing property} on $\Lambda$ if for any
 $\epsilon>0$ there is $\delta>0$ such that for any
 $\delta$-pseudo orbit $\{x_i\}_{i\in\Z}\subset\Lambda$ of $f$ there is $y\in M$
 such that $d(f^i(y), x_i)<\epsilon,$ for $i\in\Z.$
Note that in this definition, the shadowing point $y\in M$ is not
necessarily contained in $\Lambda.$ We say that $f$ belongs to the
{\it $C^1$-interior shadowing property} if there is a
$C^1$-neighborhood $\mathcal{U}(f)$ of $f$ such that for any
$g\in\mathcal{U}(f)$, $g$ has the shadowing property.

The shadowing property usually plays an important role in the
investigation of stability theory and ergodic theory(\cite{P}).

Now, we introduce the notion of the inverse shadowing property
which is a "dual" notion of the shadowing property. Inverse
shadowing property was introduced by Corless and Pilyugin in
\cite{CP}, and the qualitative theory of dynamical systems with
the property was developed by various authors(see \cite{CLZ, CP,
DLH, KOP, Le, P1}. In this paper, we introduce the various inverse
shadowing property.

Let $M^{\Z}$ be the space of all two sided sequences
$\xi=\{x_n:n\in\Z\}$ with elements $x_n\in M,$ endowed with the
product topology. For a fixed $\delta>0$, let $\Phi_f(\delta)$
denote the set of all $\delta$-pseudo orbits of $f.$ A mapping
$\varphi: M\to \Phi_f(\delta)\subset M^{\Z}$ is said to be a {\it
$\delta$-method} for $f$ if $\varphi(x)_0=x,$ and each
$\varphi(x)$ is a $\delta$-pseudo orbit of $f$ through $x,$ where
$\varphi(x)_0$ denotes the 0-th component of $\varphi(x).$ For
convenience, write $\varphi(x)$ for $\{\varphi(x)_k\}_{k\in\Z}.$
The set of all $\delta$-methods for $f$ will be denoted by
$\mathcal{T}_0(f, \delta).$ Say that $\varphi$ is {\it continuous
$\delta$-method} for $f$ if $\varphi$ is continuous. The set of
all continuous $\delta$-methods for $f$ will be denoted by
$\mathcal{T}_c(f, \delta).$ If $g:M\to M$ is a homeomorphism with
$d_0(f, g)<\delta$ then $g$ induces a continuous $\delta$-method
$\varphi_g$ for $f$ by defining
$$\varphi_g(x)=\{g^n(x):n\in\Z\}.$$ Let $\mathcal{T}_h(f, \delta)$
denote the set of all continuous $\delta$-methods $\varphi_g$ for
$f$ which are induced by a homeomorphism $g:M\to M$ with $d_0(f,
g)<\delta,$ where $d_0$ is the usual $C^0$-metric.  Let
$\mathcal{T}_d(f, \delta)$ denote by the set of all continuous
$\delta$-methods $\varphi_g$ for $f$ which are induced by
$g\in{\rm Diff}(M)$ with $d_1(f, g)<\delta.$ Then clearly we know
that
$$\mathcal{T}_d(f)\subset\mathcal{T}_h(f)\subset\mathcal{T}_c(f)\subset
\mathcal{T}_0(f),$$
$\mathcal{T}_{\alpha}(f)=\bigcup_{\delta>0}\mathcal{T}_{\alpha}(f,
\delta), \alpha=0, c, h, d.$ Let $\Lambda$ be a closed
$f$-invariant set. Denote $f|_{\Lambda}$ by  the restriction of
$f$ to a subset $\Lambda$ of $M.$ We say that $f$ has the {\it
inverse shadowing property} on $\Lambda$ with respect to the class
$\mathcal{T}_{\alpha}(f), \alpha=0, c, h, d,$ if for any
$\epsilon>0$ there exists $\delta>0$ such that for any
$\delta$-method $\varphi\in\mathcal{T}_{\alpha}(f, \delta),$ and
for a point $x\in\Lambda$ there is a point $y\in M$ such that
$$d(f^k(x), \varphi_g(y)_k)<\epsilon,\quad k\in\Z.$$
We say that $f$ has the {\it weak inverse shadowing property} on
$\Lambda$ with respect to the class $\mathcal{T}_{\alpha}(f),
\alpha=0, c, h, d,$ if for any $\epsilon>0$ there exists
$\delta>0$ such that for any
 $\delta$-method $\varphi\in\mathcal{T}_{\alpha}(f, \delta)$ and
 any point $x\in \Lambda$ there is a point $y\in M$ such that
 $$\varphi(y)\subset B_{\epsilon}(\mathcal{O}_f(x)),$$ where
 $B_{\epsilon}(A)=\{x\in M: d(x, A)\leq\epsilon\}.$ If $\Lambda=M$ then $f$ has the inverse, weak inverse shadowing property with
respect to the class $\mathcal{T}_{\alpha}(f),\alpha=0, c, h, d.$

Note that if $f\in{\rm Diff}(M)$ has the inverse shadowing
property with respect to the class $\mathcal{T}_d(f)$ then by the
definition, it clearly, has the weak inverse shadowing property
with respect to the class $\mathcal{T}_d(f).$ We say that $f$ has
the {\it orbital inverse shadowing property} on $\Lambda$ with
respect to the class $\mathcal{T}_{\alpha}(f), \alpha=0, c, h, d,
$ if for any $\epsilon>0$ there is a $\delta>0$ such that for any
$\delta$-method $\varphi\in\mathcal{T}_{\alpha}(f, \delta)$ and a
point $x\in \Lambda$ there is a point $y\in M$ such that
$$\mathcal{O}_f(x)\subset
B_{\epsilon}(\varphi(y))\quad\mbox{and}\quad \varphi(y)\subset
B_{\epsilon}(\mathcal{O}_f(x)).$$

If $\Lambda=M$ then $f$ has the orbital inverse shadowing property
with respect to the class $\mathcal{T}_{\alpha}(f),\alpha=0, c, h,
d.$

Note that if $f$ has the inverse shadowing property with respect
to the class $\mathcal{T}_d(f),$ then it has the orbital inverse
shadowing property with respect to the class $\mathcal{T}_d(f).$
But, the converse does not holds. indeed, an irrational rotation
on the unit circle has the orbital inverse shadowing property but
does not have the inverse shadowing property with respect to the
class $\mathcal{T}_d(f).$ We say that $f$ belongs to the {\it
$C^1$-interior inverse(weak inverse, or orbital inverse) shadowing
property} with respect to the class $\mathcal{T}_{\alpha}(f),
\alpha=0, c, h, d,$ if there is a $C^1$-neighborhood
$\mathcal{U}(f)$ of $f$ such that for any $g\in\mathcal{U}(f)$,
$g$ has the inverse(weak inverse, or orbital inverse) shadowing
property with respect to the class $\mathcal{T}_{\alpha}(f),
\alpha=0, c, h, d.$

 Lee \cite{Le}, showed that  a diffeomorphism belongs to the
$C^1$-interior inverse shadowing property with respect to the
$\mathcal{T}_d(f)$ if and only if it is structurally stable. And
Pilyugin \cite{P1} proved that a diffeomorphism belongs to the
$C^1$-interior inverse shadowing property with respect to the
class $\mathcal{T}_c(f)$ if and only if it is structurally stable.
Thus we can restate the above facts as follows.

\begin{thm}\label{thm1}Let $f\in{\rm Diff}(M).$ A diffeomorphism $f$ belongs to the $C^1$-interior inverse
shadowing property with respect to the class
$\mathcal{T}_d(f)$[resp. $\mathcal{T}_c(f)$] if and only if it is
structurally stable.
\end{thm}

In \cite{CLZ} Choi, Lee and Zhang  showed that a diffeomorphism
belongs to the $C^1$-interior weak inverse shadowing property with
respect to the $\mathcal{T}_d(f)$ if and only if both Axiom A and
the no-cycle condition. Moreover, they proved that a
diffeomorphism belongs to the $C^1$-interior orbital inverse
shadowing property with respect to the class $\mathcal{T}_d(f)$ if
and only if both Axiom A and the strong transversal condition.
 From the above facts, we get
the follows.
\begin{thm}Let $f\in{\rm Diff}(M).$ If a diffeomorphism $f$ belongs to the $C^1$-interior
weak inverse shadowing property with respect to the class
$\mathcal{T}_d(f)$ then $f$ satisfying both Axiom A and the
no-cycle condition. Moreover, if $f$ belongs to the $C^1$-interior
orbital inverse shadowing property with respect to the class
$\mathcal{T}_d(f)$ then it is structurally stable.
\end{thm}

By the theorem, even though a diffeomorphism is contained in the
$C^1$-interior of the set of diffeomorphisms possessing the weak
inverse shadowing property with respect to the class
$\mathcal{T}_d(f)$, it does not necessarily satisfy the strong
transversality condition.

A periodic point $p$ of $f$ is {\it hyperbolic} if $Df^{\pi(p)}$
has eigenvalues with absolute values different of one, where
$\pi(p)$ is the period of $p.$ Denote by $\mathcal{F}(M)$ the set
of $f\in{\rm Diff}(M)$ such that there is a $C^1$-neighborhood
$\mathcal{U}(f)$ of $f$ such that for any $g\in\mathcal{U}(f),$
every $p\in P(g)$ is hyperbolic. It is proved that by Hayashi
\cite{H} that $f\in\mathcal{F}(M)$ if and only if $f$ satisfies
both Axiom A and the no-cycle condition.

Let $\Lambda$ be a closed $f\in{\rm Diff}(M)$-invariant set. We
say that $\Lambda$ is {\it hyperbolic} if the tangent bundle
$T_{\Lambda}M$ has a $Df$-invariant splitting $E^s\oplus E^u$ and
there exists constants  $C>0$ and $0<\lambda<1$ such that
$$\|D_xf^n|_{E_x^s}\|\leq C\lambda^n\;\;{\rm and}\;\;\|D_xf^{-n}|_{E_x^u}\|\leq C\lambda^{n} $$
for all $x\in \Lambda$ and $n\geq 0.$ If $\Lambda=M$ then we say
that $f$ is an {\it Anosov diffeomorphism}.

\section{Statement of the results}

A fundamental problem in differentiable dynamical systems is to
understand how a robust dynamic property on the underlying
manifold would influence the behavior of the tangent map on the
tangent bundle. For instance, in \cite{M}, Ma\~n\'e proved that
any $C^1$ structurally stable diffeomorphism is an Axiom A
diffeomorphism. And in \cite{Pa}, Palis extended this result to
$\Omega$-stable diffeomorphisms.

 Let $M$ be a compact $C^{\infty}$
$n$-dimensional Riemannian manifold endowed with a volume form
$\omega.$ Let $\mu$ denote the measure associated to $\omega,$
that we call Lebesgue measure, and let $d$ denote the metric
induced by the Riemannian structure. Denote by ${\rm
Diff}_{\mu}(M)$ the set of diffeomorphisms which preserves the
Lebesgue measure $\mu$ endowed with the $C^1$-topology. In the
volume preserving, the Axiom A condition is equivalent to the
diffeomorphism be Anosov, since $\Omega(f)=M$ by Poincar\'e
Recurrence Theorem. The purpose of this paper is to do this using
the robust property.

We define the set $\mathcal{F}_{\mu}(M)$ as the set of
diffeomorphisms $f\in{\rm Diff}_{\mu}(M)$ which have a
$C^1$-neighborhood $\mathcal{U}(f)\subset{\rm Diff}_{\mu}(M)$ such
that if for any $g\in\mathcal{U}(f)$, every periodic point of $g$
is hyperbolic. Note that
$\mathcal{F}_{\mu}(M)\subset\mathcal{F}(M)$(see \cite[Corollary
1.2]{AC}).

Very recently, Arbieto and Catalan \cite{AC} proved that if a
volume preserving diffeomorphism contained in
$\mathcal{F}_{\mu}(M)$ then it is Anosov. Indeed, the first they
used the Ma\~n\'e's results(\cite[Proposition II.1]{M}). Then they
showed that $\overline{P(f)}$ is hyperbolic. And, they proved that
nonwandering set $\Omega(f)=\overline{P(f)}$ by Pugh's closing
lemma. Finally, by Pincar\'e 's Recurrence Theorem, $\Omega(f)=M.$
 From the above facts, we can restate as follows.

\begin{thm}\label{thm1} Any diffeomorphism in
$\mathcal{F}_{\mu}(M)$ is Anosov.
\end{thm}

In \cite{L}, Lee showed that if a volume preserving
diffeomorphisms belongs to the $C^1$-interior expansive or
$C^1$-interior shadowing property, then it is Anosov. And
\cite{L1} proved that if a volume preserving diffeomorphisms
belongs to the $C^1$-interior weak shadowing property or
$C^1$-interior weak limit shadowing property, then it is Anosov.
Form this results,  we study the cases when a volume preserving
diffeomorphism $f$ is in $C^1$-interior various inverse shadowing
property with respect to the class $\mathcal{T}_d(f)$, then it is
Anosov. Let $int\mathcal{IS}_{\mu}(M)$ be denote the set of volume
preserving diffeomorphisms in ${\rm Diff}_{\mu}(M)$ satisfying the
inverse shadowing property with respect to the class
$\mathcal{T}_d$, and let $int\mathcal{WIS}_{\mu}(M)$[respect.
$int\mathcal{OIS}_{\mu}(M)$] be denote the set of
volume-preserving diffeomorphisms in ${\rm Diff}_{\mu}(M)$
satisfying the  weak inverse shadowing property with respect to
the class $\mathcal{T}_d$ [respect. the orbital inverse shadowing
property with respect to the class $\mathcal{T}_d$] . From now, we
only consider the class $\mathcal{T}_d$ when we mention the
inverse shadowing property; that is, the "inverse shadowing
property" implies the " inverse shadowing property with respect to
the class $\mathcal{T}_d$". Now we are in position to state the
theorem of our paper.

\begin{thm}\label{thm2} Let $f\in{\rm Diff}_{\mu}(M).$ We has that
$$int \mathcal{IS}_{\mu}(M)= int\mathcal{OIS}_{\mu}(M)=
int\mathcal{WIS}_{\mu}(M)=\mathcal{AN}_{\mu}(M),$$ where
$\mathcal{AN}_{\mu}(M)$ is the set of Anosov volume preserving
diffeomorphisms in ${\rm Diff}_{\mu}(M)$.
\end{thm}

\section{Proof of Theorem \ref{thm2}}

 Let $M$ be a compact $C^{\infty}$
$n$-dimensional Riemannian manifold endowed with a volume form
$\omega$, and let $f\in{\rm Diff}_{\mu}(M).$
 To prove the results, we will use the following is the well-known
 Franks' lemma for the conservative case, stated and proved in
 \cite[Proposition 7.4]{BDP}.
 \begin{lem}\label{frank} Let $f\in {\rm Diff}^1_{\mu}(M)$, and
 $\mathcal{U}$ be a $C^1$-neighborhood of $f$ in ${\rm
 Diff}^1_{\mu}(M).$ Then there exist a $C^1$-neighborhood
 $\mathcal{U}_0\subset \mathcal{U}$ of $f$ and $\epsilon>0$ such
 that if $g\in\mathcal{U}_0$, any finite $f$-invariant set $E=\{x_1, \ldots,
 x_m\},$ any neighborhood $U$ of $E$ and any volume-preserving
 linear maps $L_j:T_{x_j}M\to T_{g(x_j)}M$ with
 $\|L_j-D_{x_j}g\|\leq\epsilon$ for all $j=1,\ldots, m,$ there is a
 conservative diffeomorphism $g_1\in\mathcal{U}$ coinciding with
 $f$ on $E$ and out of $U,$ and $D_{x_j}g_1=L_j$ for all $j=1,
 \ldots, m.$
\end{lem}

\begin{rk}\label{moser} Let $f\in{\rm Diff}^1_{\mu}(M).$ From the Moser's
Theorem(see \cite{Mo}), there is a smooth conservative change of
coordinates $\varphi_x:U(x)\to T_xM$ such that $\varphi_x(x)=0,$
where $U(x)$ is a small neighborhood of $x\in M.$
\end{rk}
\begin{pro}\label{is} If $f\in int\mathcal{IS}_{\mu}(M),$ then
every periodic point of $f$ is hyperbolic.
\end{pro}
\begin{proof}Take $f\in int\mathcal{IS}_{\mu}(M),$ and
$\mathcal{U}(f)$ a $C^1$-neighborhood of $f\in
int\mathcal{E}_{\mu}(M).$ Let $\epsilon>0$ and
$\mathcal{V}(f)\subset\mathcal{U}_0(f)$ corresponding number and
$C^1$-neighborhood given by Lemma \ref{frank}. To derive a
contradiction, we may assume that there exists a nonhyperbolic
periodic point $p\in P(g)$ for some $g\in\mathcal{V}(f).$ To
simplify the notation in the proof, we may assume that $g(p)=p.$
Then there is at least one eigenvalue $\lambda$ of $D_pg$ such
that $|\lambda|=1.$

 By making use of the Lemma \ref{frank}, we
linearize $g$ at $p$ with respect to Moser's Theorem; that is, by
choosing $\alpha>0$ sufficiently small we construct $g_1$
$C^1$-nearby $g$ such that
$$g_1(x)=\left\{%
\begin{array}{ll}
    \varphi^{-1}_p\circ D_pg\circ \varphi_p(x) &\mbox{if} \quad x\in B_{\alpha}(p), \\
    g(x) & \mbox{if}\quad x\notin B_{4\alpha}(p). \\
\end{array}%
\right.$$ Then $g_1(p)=g(p)=p.$

First, we may assume that $\lambda\in\R$. Let $v$ be the
associated non-zero eigenvector such that $\|v\|=\alpha/4.$ Then
we can get a small arc $\mathcal{I}_v=\{tv:-1\leq t\leq 1\}\subset
\varphi_p(B_{\alpha}(p)).$ Take $\epsilon_1=\alpha/8$. Let
$0<\delta<\epsilon_1$ be the number of the inverse shadowing
property of $g_1$ for $\epsilon_1.$ Then by our construction of
$g_1,$ $\varphi^{-1}_p(\mathcal{I}_v)\subset B_{\alpha}(p).$ Put
$\mathcal{J}_p=\varphi_p(\mathcal{I}_v).$ For the above
$\delta>0$, we can define $\mathcal{T}_d(g_1)$-method as follows;
Let $h\in{\rm Diff}_{\mu}(M)$ such that $$h(x)=(x_1+\delta,
Ax')\quad\mbox{and}\quad h^{-1}(x)=(x_1+\delta, A^{-1}x'),$$ where
$x'=(x_2,\ldots, x_n)$ and $A$ corresponding to $|\lambda|\not=1.$
Clearly, $d_1(g_1, h)<\delta,$ and $h\in\mathcal{T}_d(g_1).$ Let
$p=0.$ Then choose $x=(x_1, 0, \ldots, 0)\in\mathcal{J}_p$ such
that $d(0, x)=2\epsilon_1.$ Then $$d(h^i(0), g_1^i(x))=d(0,
x)=2\epsilon_1.$$ Thus $g_1$ does not have the inverse shadowing
property.

We take a point $y=(y_1, 0\ldots, 0)\in\mathcal{J}_p$ such that
$d(y, x)<\epsilon_1.$ Then $d(g^i(x), h^i(y))=d(x, y_1+i\delta)$
and for some $k\in\Z$
$$d(x, y_1+k\delta)>\alpha/8=\epsilon_1.$$ Thus $g_1$ does not have
the inverse shadowing property.

 Therefore, we can choose a point
$y\in M\setminus\mathcal{J}_p$ such that $d(x, y)<\epsilon_1.$
Since $d(g_1^i(x), h^i(y))=d((x_1,0\ldots, 0),
(y_1+i\delta,A^ix'))$ or

$d((x_1, 0,\ldots,0), h^i(x))=((x_1, 0,\ldots,0),
(x_1+i\delta,(A^{-1})^ix')),$ we can find $k\in\Z$ such that
\begin{align*}&d((x_1, 0,\ldots,0),
(x_1+k\delta,A^kx'))>\epsilon_1\quad\mbox{or}\\& d((x_1,
0,\ldots,0), (x_1+k\delta,(A^{-1})^kx'))>\epsilon_1. \end{align*}
Thus $g_1$ does not have the inverse shadowing property. This is a
contradiction since $f\in int\mathcal{IS}_{\mu}(M).$

Finally,  if $\lambda\in\C,$ then to avoid the notational
complexity, we may assume that $g(p)=p.$ As in the first case, by
Lemma \ref{frank}, there are $\alpha>0$ and $g_1\in
\mathcal{V}(f)$ such that $g_1(p)=g(p)=p$ and

$$g_1(x)=\left\{%
\begin{array}{ll}
    \varphi^{-1}_p\circ D_pg\circ \varphi_p(x) &\mbox{if} \quad x\in B_{\alpha}(p), \\
    g(x) & \mbox{if}\quad x\notin B_{4\alpha}(p). \\
\end{array}%
\right.$$

With a $C^1$-small modification of the map $D_pg$, we may suppose
that there is $l>0$(the minimum number) such that $D_pg^l(v)=v$
for any $v\in\varphi_p(B_{\alpha}(p))\subset T_pM.$ Then, we can
go on with the previous argument in order to reach the same
contradiction. Thus, every periodic point of $f\in
int\mathcal{IS}_{\mu}(M)$ is hyperbolic.
\end{proof}

\begin{pro}\label{wis} If $f\in int\mathcal{WIS}_{\mu}(M)$ then
every periodic point of $f$ is hyperbolic.
\end{pro}
\begin{proof}Take $f\in int\mathcal{WIS}_{\mu}(M),$ and
$\mathcal{U}(f)$ a $C^1$-neighborhood of $f\in
int\mathcal{WIS}_{\mu}(M).$ Let $\epsilon>0$ and
$\mathcal{V}(f)\subset\mathcal{U}_0(f)$ corresponding number and
$C^1$-neighborhood given by Lemma \ref{frank}. To derive a
contradiction, we may assume that there exists a nonhyperbolic
periodic point $p\in P(g)$ for some $g\in\mathcal{V}(f).$ To
simplify the notation in the proof, we may assume that $g(p)=p.$
Then as in the proof of Proposition \ref{is}, we can take
$\alpha>0$ sufficiently small, and a smooth map $\varphi_p:
B_{\alpha}(p)\to T_pM$. Then we can make an arc
$\mathcal{J}_p\subset B_{\alpha}(p)$ and for some
$g_1\in\mathcal{V}(f)$. Take $\epsilon_1=({\rm
length}{\mathcal{J}_p})/4.$ Let $0<\delta<\epsilon_1$ be the
number of the weak inverse shadowing property of $g_1$ for
$\epsilon_1.$ Then we can construct a map $h\in{\rm Diff}(M)$ as
in the proof of Proposition\ref{is}. Let $p=0.$ Then choose a
point $x=(x_1, 0,\ldots, 0)\in\mathcal{J}_p$ such that $d(0,
x)=2\epsilon_1.$ Since $g_1$ has the weak inverse shadowing
property,
$$\mathcal{O}_h(0)\subset B_{\epsilon_1}(\mathcal{O}_{g_1}(x)).$$
However, for any $y\in\mathcal{J}_p,$
$$g_1^i(y)=y\quad\mbox{and}\quad h^i(y)=(y_1+i\delta, A^iy'),$$
where $y'=(y_2, \ldots, y_n).$ Thus it easily see that
$$\mathcal{O}_h(0)\not\subset
B_{\epsilon_1}(\mathcal{O}_{g_1}(x)).$$

If $y=(y_1, 0,\ldots, 0)\in\mathcal{J}_p\setminus\{p\}$, then
$$h^k(y)=(y_1+k\delta, A^ky')=(y_1+k\delta, 0).$$
Thus we know that
$$\mathcal{O}_h(y)\not\subset B_{\epsilon_1}(\mathcal{O}_{g_1}(x)).$$
This is a contradiction.

 Finally, we can choose a point $y\in
M\setminus\mathcal{J}_p$ such that $d(x, y)<\epsilon_1.$ Then we
know that
$$h^k(y)=(y_1+k\delta, A^ky').$$ Therefore, $$h^l(y)\not\in
B_{\epsilon_1}(\mathcal{J}_p),$$ for some $l\in\Z.$ Then
$$\mathcal{O}_h(y)\not\subset
B_{\epsilon_1}(\mathcal{O}_{g_1}(x))=B_{\epsilon_1}(x).$$ Thus
$g_1$ does not have the weak inverse shadowing property. This is a
contradiction.

If $\lambda\in\C,$ then as in the proof of Proposition \ref{is},
for $g_1\in\mathcal{V}(f),$ we can take $l>0$ such that
$D_pg^l_1(v)=v$ for any $v\in\varphi_p(B_{\alpha}(p))\subset
T_pM.$ Then from the previous argument in order to reach the same
contradiction. Thus, every periodic point of $f\in
int\mathcal{WIS}_{\mu}(M)$ is hyperbolic.

Consequently, if $f\in int\mathcal{WIS}_{\mu}(M)$ then
$f\in\mathcal{F}_{\mu}(M).$

\end{proof}

\begin{pro}\label{ois} If $f\in int\mathcal{OIS}_{\mu}(M)$ then
every periodic point of $f$ is hyperbolic.
\end{pro}
\begin{proof} The proof is almost the same that of Proposition
\ref{wis}. Indeed, let $f\in int\mathcal{OIS}_{\mu}(M),$ and
$\mathcal{U}(f)$ a $C^1$-neighborhood of $f\in
int\mathcal{OIS}_{\mu}(M).$ Let $\epsilon>0$ and
$\mathcal{V}(f)\subset\mathcal{U}_0(f)$ corresponding number and
$C^1$-neighborhood given by Lemma \ref{frank}. To derive a
contradiction, we may assume that there exists a nonhyperbolic
periodic point $p\in P(g)$ for some $g\in\mathcal{V}(f).$ To
simplify the notation in the proof, we may assume that $g(p)=p.$
Then as in the proof of Proposition \ref{wis}, we can take
$\alpha>0$ sufficiently small, and a smooth map $\varphi_p:
B_{\alpha}(p)\to T_pM$. Then we can make an arc
$\mathcal{J}_p\subset B_{\alpha}(p)$ and for some
$g_1\in\mathcal{V}(f)$. Take $\epsilon_1=({\rm
length}{\mathcal{J}_p})/4.$ Let $0<\delta<\epsilon_1$ be the
number of the orbital inverse shadowing property of $g_1$ for
$\epsilon_1.$ Then we can construct a map $h\in{\rm Diff}(M)$ as
in the proof of Proposition\ref{wis}. Let $p=0.$ Then choose a
point $x=(x_1, 0,\ldots, 0)\in\mathcal{J}_p$ such that $d(0,
x)=2\epsilon_1.$ Since $g_1$ has the orbital inverse shadowing
property,
\begin{align*}\mathcal{O}_h(0)\subset
B_{\epsilon_1}(\mathcal{O}_{g_1}(x))\quad\mbox{and}\quad\mathcal{O}_{g_1}(x)\subset
B_{\epsilon_1}(\mathcal{O}_h(0)).\end{align*} However, for any
$y\in\mathcal{J}_p,$
$$g_1^i(y)=y\quad\mbox{and}\quad h^i(y)=(y_1+i\delta, A^iy'),$$
where $y'=(y_2, \ldots, y_n).$ Thus it easily see that
$$\mathcal{O}_h(0)\not\subset
B_{\epsilon_1}(\mathcal{O}_{g_1}(x)).$$ Thus $g_1$ does not have
the orbital inverse shadowing property.

If $y=(y_1, 0,\ldots, 0)\in\mathcal{J}_p\setminus\{p\}$, then
$$h^k(y)=(y_1+k\delta, A^ky')=(y_1+k\delta, 0).$$
Thus we know that
$$\mathcal{O}_h(y)\not\subset B_{\epsilon_1}(\mathcal{O}_{g_1}(x)).$$
This is a contradiction.

 Finally, we can choose a point $y\in
M\setminus\mathcal{J}_p$ such that $d(x, y)<\epsilon_1.$ Then we
know that
$$h^k(y)=(y_1+k\delta, A^ky').$$ Therefore, $$h^l(y)\not\in
B_{\epsilon_1}(\mathcal{J}_p),$$ for some $l\in\Z.$ Then
$$\mathcal{O}_h(y)\not\subset
B_{\epsilon_1}(\mathcal{O}_{g_1}(x))=B_{\epsilon_1}(x).$$ Thus
$g_1$ does not have the orbital inverse shadowing property. This
is a contradiction.

If $\lambda\in\C,$ then as in the proof of Proposition \ref{wis},
for $g_1\in\mathcal{V}(f),$ we can take $l>0$ such that
$D_pg^l_1(v)=v$ for any $v\in\varphi_p(B_{\alpha}(p))\subset
T_pM.$ Then from the previous argument in order to reach the same
contradiction. Thus, every periodic point of $f\in
int\mathcal{OIS}_{\mu}(M)$ is hyperbolic.
\end{proof}



\bigskip
\bigskip


\bigskip
\begin{thebibliography}{11}


\bibitem{AC}
A. Arbieto and T. Catalan, \emph{ Hyperbolicity in the volume
preserving senario}, prerpint.
\bibitem{BDP}
C. Bonatti, L. J. Di\'az and E. R. Pujals, \emph{A $C^1$-generic
dichotomy for diffeomorphism: weak forms of hyperbolicity or
infinitely many sinks or sources}, Ann. of Math., {\bf 116}(2003),
355-418.
\bibitem{CLZ}
T. Choi, K. Lee and Y. Zhang, \emph{Chracterisations of
$\Omega$-stability and structral stability via inverse shadowing},
Bull. Austral. Math. Soc., {\bf 74}(2006), 185-196.
\bibitem{CP}
R. Corless and S. Pilyugin, \emph{Approximate and real
trajectories for generic dynamical systems}, J. Math. Anal. Appl.,
{\bf189}(1995), 409-423.
\bibitem{DLH}
P. Diamond, K. Lee and Y. Han, \emph{Bishadowing and
hyperbolicity}, Internat. J. Bifur. Chaos appl. Sci. Engr.
{\bf12}(2002), 1779-1788.

\bibitem{H}
S. Hayashi,\emph{ Diffeomorphisms in $\mathcal{F}^1(M)$ satisfy
Axiom A }, Ergodic Thoery $\&$ Dynam. Syst., {\bf 12} (1992),
233-253.

\bibitem{KOP}
P. Kloeden, J. Ombach and A. Porkrovskii, \emph{Continuous and
inverse shadowing}, Funct. Differ. Equa., {\bf6}(1999), 137-153.

\bibitem{Le}
K. Lee, \emph{Continuous inverse shadowing and hyperbolicity},
Bull. Austral. Math. Soc., {\bf 67}(2003), 15-26.
\bibitem{L}
M. Lee, \emph{Volume preserving diffeomorphisms with expansive and
shadowing}, preprint.
\bibitem{L1}
M. Lee, \emph{Volume preserving diffeomorphisms with weak and weak
limit shadowing}, preprint.

\bibitem{M}
R. Ma\~n\'e,\emph{ A proof of the $C^1$-stability conjecture}.
Publ. Math. de IHES, {\bf 66}(1987),161-210.
\bibitem{Mo}
J. Moser, \emph{On the volume elements on a manifold}, Trans.
Amer. Math. Soc.,{\bf 120}(1965), 286-294.
\bibitem{Pa}
J. Palis,\emph{On the $C^1$ $\Omega$-stability conjecture}. Inst.
Hautes tudes Sci. Publ. Math.,{\bf66}(1988),211-215.
\bibitem{P}
S. Pilyugin, "Shadowing in Dynamical Systems", Lecture Notes in
Math. Springer-Verlag, {\bf 1706}(1999).
\bibitem{P1}
S. Pilyugin, \emph{Inverse shadowing by continuous methods}, Disc.
Contin. Dynam. Syst., {\bf8} (2002), 29-38.
 \end{thebibliography}
\end{document}